\documentclass[11pt, a4 paper]{amsart} 

\usepackage[psamsfonts]{amssymb} 
\usepackage{amsfonts,amsmath,mathabx}
\usepackage{color}

\usepackage{amsthm, amssymb}
\usepackage{amsfonts}

\newtheorem{thmA}{Theorem}

\numberwithin{equation}{section} 

\newtheorem{theorem}{Theorem}[section]
\newtheorem{prop}[theorem]{Proposition}
\newtheorem{proposition}[theorem]{Proposition}
\newtheorem{lemma}[theorem]{Lemma}
\newtheorem{corollary}[theorem] {Corollary}

\newtheorem*{theorem*}{1-2-3 Theorem} 

\theoremstyle{remark}

\theoremstyle{definition}

\def\F{\rm{F}} 
\def\Z{\mathbb Z}
\def\N{\mathbb N}
\def\R{\mathbb R}

\def\G{\Gamma}

\def\-{\overline}
\def\wh{\widehat}

\def\Fix{{\rm{Fix}}}

\def\cat{{\rm{CAT}}${\rm{(0)}}$ }

\def\Min{{\rm{Min}}}

\def\isom{{\rm{Isom}}}

\def\fix{{\rm{Fix}}}

\def\script{\mathfrak}

\def\G{\Gamma} 

\def\G{\Gamma}  
\def\g{\gamma}

\def\<{\langle}
\def\>{\rangle}


\begin{document}

\title[Profinite isomorphisms and fixed-point properties]{Profinite isomorphisms and fixed-point properties}

\author[Martin R. Bridson]{Martin R.~Bridson}
\address{Mathematical Institute\\
Andrew Wiles Building\\
Woodstock Road\\
Oxford}
\email{bridson@maths.ox.ac.uk}




\begin{abstract}  We describe a flexible construction that produces triples of finitely generated, residually finite
groups $M\hookrightarrow P \hookrightarrow \G$, where the maps induce isomorphisms of profinite completions
$\wh{M}\cong\wh{P}\cong\wh{\G}$, but $M$ and $\G$ have Serre's property FA while $P$ does not.
In this construction, $P$ is finitely presented and $\G$ is of type ${\rm{F}}_\infty$. More generally, given any positive
integer $d$, one can demand that $M$ and $\G$ have a fixed point whenever they act by semisimple isometries on
a complete CAT$(0)$ space of dimension at most $d$, while $P$ acts without a fixed point on a tree.
\end{abstract}

\maketitle

\section{Introduction}                  

In the quest for profinite invariants of discrete groups, fixed-point properties have been a source of disappointment. For example, Aka  \cite{aka} proved that the profinite completion of a finitely generated, residually finite group does not determine whether the group has property (T), i.e.~whether
 the group can act without a global fixed point as a group of affine isometries of a Hilbert space. 
Cheetham-West, Lubotzky, Reid and Spitler  \cite{tam} proved a similar theorem for actions on trees:
they construct pairs of finitely presented, residually finite groups $G_1$ and $G_2$ such that $\wh{G}_1\cong \wh{G}_2$  but $G_1$ has Serre's property FA whereas $G_2$ does not. 
(Here, $\wh{G}_i$ denotes the profinite completion on $G_i$.)

In the present paper, we will  improve upon this last result in two ways. First, we construct groups with these properties 
for which  $(G_2, G_1)$ is a Grothendieck pair,
i.e.~the isomorphism $\wh{G}_1\cong \wh{G}_2$ is induced by a monomorphism of discrete groups 
$G_1 \hookrightarrow G_2$ (cf.~\cite{tam} Question 4.1). 
Secondly, we extend this result from actions on trees (the 1-dimensional case) to actions on 
$d$-dimensional \cat spaces,  with $d\ge 1$ arbitrary. 

We say that a group $G$ has {\em property $\Fix_d$} if $G$ fixes a point
whenever it acts by semsimple isometries on a complete $\cat$ space of dimension at most $d$.
Every isometry of a simplicial tree is semisimple, so $\Fix_1$ implies Serre's property FA (and extends it to cover
actions on complete $\R$-trees).

\begin{thmA}\label{t:main}
For every integer $d\ge 1$, there exist triples of  residually finite
groups $M\overset{i}\hookrightarrow P\overset{j}\hookrightarrow \G$ such that:
\begin{enumerate}
\item $i$ and $j$ induce isomorphisms $\wh{M}\cong\wh{P}\cong\wh{\G}$;
\item $M$ is finitely generated, $P$ is finitely presented, and $\G$ is of type $F_\infty$;
\item $M$ and $\G$ have property $\Fix_d$, but 
\item $P$ is a non-trivial amalgamated free-product and therefore acts on a simplicial tree without a global fixed point. 
\end{enumerate}
\end{thmA}

An artefact of our proof is that although $M$ and $\Gamma$ have $\Fix_d$, they each contain a subgroup of finite index that can act on a tree without fixing a point.

The fixed-point properties required in the above theorem will be established using the following criterion, which is drawn from 
the circle of ideas developed in \cite{mb:mcg} and \cite{mb:helly}. 

\begin{thmA}\label{t:thmB}
If $A$ is a finitely generated group with finite abelianisation and $B$ is a finite group, then
$A\wr B$ has $\Fix_d$ for $d= |B|-1$.
\end{thmA}

The first steps in our construction  of the triples $M\hookrightarrow P\hookrightarrow \G$ in Theorem \ref{t:main} follow the template for constructing finitely presented
Grothendieck pairs that originates in  \cite{BG} and is explicit in  Section 8 of \cite{mb:pcmi}. We craft  
finitely presented groups $Q$ that enjoy an array of properties  
relevant to our aims (Section \ref{s:input});
we use a suitably-adapted form of the Rips construction (Proposition \ref{p:rips}) to produce short exact sequences $1\to N\to G \to Q\to 1$ with $G$  finitely presented and residually finite, $N$ finitely generated, and both $N$ and $G$ perfect; 
and we  take a fibre product of several copies of $G \to Q$ to produce $N^d\hookrightarrow P_d\hookrightarrow G^d$ with $P_d$ finitely presented. (A novel feature here is that we take the fibre product of several copies of $G \to Q$, not just two.)
The  triples $M\overset{i}\hookrightarrow P\overset{j}\hookrightarrow \G$  we seek
 are obtained by taking  finite extensions of $N^d, P_d$ and $G^d$ in a way that allows 
 us to apply Theorem~\ref{t:thmB}.  
 
There is a great deal of flexibility in this construction  -- see Section \ref{s:last}.

\section{Preliminaries}

In this section we gather the basic definitions and facts we need concerning profinite completions of groups and isometries of $\cat$ spaces.

\subsection{Profinite completions}

If $M_1<M_2$ are normal 
subgroups of finite index in a group $G$, then there is a natural map $G/M_1\to G/M_2$. Thus the
finite quotients of $G$ form a directed system.
The {\em profinite completion} of $G$ is the inverse limit of this system:
$$
\wh{G} := \lim_{\leftarrow} G/M.
$$ 
 The natural map   $i: G\to\wh{G}$  is injective if and only if $G$ is residually finite. If $G$
 is finitely generated then, for every finite group $Q$, composition with $i$ defines a bijection ${\rm{Hom}}(\wh{G}, Q)
 \to {\rm{Hom}}({G}, Q)$ that restricts to a bijection on the set of epimorphisms.
 In particular, $G$ and $\wh{G}$ have the same set of finite  images, which we denote by $\script{F}(G)$. 
 Thus $\wh{G}_1\cong \wh{G}_2$ implies $\script{F}(G_1)= \script{F}(G_2)$. Less obviously, 
 for finitely
generated groups, $\script{F}(G_1)= \script{F}(G_2)$ implies $\wh{G}_1\cong \wh{G}_2$ -- see  \cite[pp.~88--89]{RZ}.
 (Note that $\wh{G}_1\cong \wh{G}_2$ does not imply that there are any non-trivial homomorphisms ${G}_1\to G_2$.)

A property $\script{P}$ of finitely generated, residually finite groups is said to be a {\em profinite invariant} if
 $\wh{G}_1\cong \wh{G}_2$ implies that $G_2$ has $\script{P}$ whenever $G_1$ has $\script{P}$. 
 Theorem \ref{t:main} shows that $\Fix_d$ is not a profinite invariant. 

A pair of finitely generated, residually finite groups $G_1\overset{\iota}\hookrightarrow G_2$ is called a {\em Grothendieck pair}
\cite{groth} if  the induced map $\wh{\iota}: \wh{G}_1\to \wh{G}_2$ is an isomorphism.  For fixed 
$G_2$, there can be infinitely many non-isomorphic subgroups $G_1$ such that $G_1\hookrightarrow G_2$ is a Grothendieck
pair, even if one requires both $G_1$ and $G_2$ to be finitely presented \cite{mb:jems}.

\subsection{Isometries of \cat spaces}

We refer the reader to \cite{BH} for basic facts about \cat spaces.  
We write $\isom (X)$ for the group of isometries of a \cat space $X$ and $\fix (H)$ for the set of points in $X$ fixed by each element of a subset $H\subset \isom (X)$. Note that $\fix (H)$ is closed and convex.

If $X$ is complete, each closed, non-empty bounded subset is contained in a unique smallest ball; see \cite{BH} page 178. If the bounded subset is an orbit of a subgroup $H<\isom(X)$, then the centre of the ball will be fixed by $H$. This proves the
following standard proposition.

\begin{prop}\label{l:finite}
If $X$ is a complete \cat space, then every finite subgroup of $\isom(X)$ fixes a point in $X$.
\end{prop}
 
By combining  the preceding bounded-orbit observation with the fact that $\fix (H)$
is itself a \cat space,  
one can prove the following standard fact  -- see \cite[Corollary 2.5]{mb:mcg}, for example.

\begin{prop}\label{c2.5}
Let $X$  be a complete \cat space. 
If the subgroups $H_1,\dots, H_n<\isom(X)$ commute and $\fix(H_i)$ is non-empty for $i=1,\dots,n$, then $\bigcap_i\fix (H_i)$ is non-empty.
\end{prop}

For an isometry 
$\gamma\in\isom (X)$,
$$\Min (\gamma) := \{ p\in X \mid d(p,\gamma.p) = |\gamma|\},$$
where $|\g| := \inf\{ d(x,\g.x) \mid x\in X\}$. By definition, $\g$ is {\em semisimple} if
$\Min (\gamma)$ is non-empty. Every isometry of a complete $\R$-tree is semisimple.  
Semisimple isometries are divided into {\em hyperbolics} (also called loxodromics), for which
$|\g |>0$, and {\em elliptics}, which are the isometries with $\fix (\gamma)\neq\emptyset$. 
If $\g$ is hyperbolic then there exist $\g$-invariant isometrically embedded 
lines $\R\hookrightarrow X$ on which $\g$ acts as a translation by $|\g|$; each such line is called an axis for $\g$. The union of these axes is $\Min(\g)$.
 The following extract from pages 229--231 of \cite{BH}  summarizes the
properties of $\Min (\gamma)$ that we require.

\begin{prop}\label{p:min}
Let $X$ be a complete \cat space and let $\g\in\isom(X)$ be a hyperbolic isometry. Then,
\begin{enumerate}
\item $\Min (\gamma)$ splits isometrically $\Min (\gamma)= Y\times \R$, where $Y\times\{0\}$
 is a closed, convex subspace of $X$;
 \item $\g$ acts trivially on $Y$ and acts as translation by $|\g|$ on each of the lines $\{y\}\times\R$;
 \item the centraliser $C(\g)<\isom(X)$ leaves $\Min (\gamma)$ and its splitting invariant,
 acting by translations of the second factor;
 \item if $\delta\in C(\g)$ is hyperbolic, then $\Min (\gamma)$ contains an axis for $\delta$.
\end{enumerate}
\end{prop}

\section{A Rips construction and input groups}

The purpose of this section is to produce the short exact sequences $1\to N\to G\to Q\to 1$ described in the
introduction.

\subsection{The input groups $Q$}\label{s:groups}\label{s:input}

Our constructions require as input a group $Q$ with the following properties:
\begin{enumerate}
\item [$\bullet$] $Q$ is of type $\F_\infty$ (i.e.~has a classifying space $K(Q,1)$ with finite skeleta);
\item [$\bullet$] $H_2(Q,\Z)=0$;
\item [$\bullet$] $\wh{Q}=1$;
\item [$\bullet$] $Q$ is a non-trivial amalgamated free product (and therefore does not have FA).
\end{enumerate}

There are many  ways to concoct groups $Q$ with these properties. Indeed, {\em every} finitely presented group  
can be embedded (explicitly, with controlled geometry) into a  finitely presented group
 that has no non-trivial finite quotients
\cite{mb:embed};
by replacing this enveloping group with its universal central extension one can  force it to have trivial second homology;
and by taking a free product of two copies of the resulting group one obtains a group $Q$ satisfying all of the above
properties. If the group that one starts with is $F_\infty$ (resp. type $F$), then so is $Q$.

One can also find explicit groups of the desired form in the literature. For example,  from
\cite{BG} one could take
$$ 
Q= \< a, b, \alpha, \beta \mid ba^{-p}b^{-1}a^{p+1},\, \beta\alpha^{-p}\beta^{-1}\alpha^{p+1},\, 
[bab^{-1},a]\beta^{-1},\, [\beta\alpha\beta^{-1},\alpha]b^{-1}\>.
$$

\subsection{A convenient version of the Rips construction}\label{s:rips} 

There are many refinements 
of the Rips construction in the literature, with various properties imposed on the groups constructed.
The following version suits our needs.

\begin{prop}\label{p:rips}
There exists an algorithm that, given 
a finite presentation $ \< X\mid R\>$ of a group $Q$,
will construct a finite aspherical presentation $ \< X\cup \{a_1,a_2\}\mid \widetilde{R}\cup V\>$
for a group $G$ so that:
\begin{enumerate}
\item $G$ is hyperbolic and residually finite;
\item $N:=\<a_1,a_2\>$ is normal in $G$;
\item $G/N$ is isomorphic to $Q$; 
\item $G$ is perfect if $Q$ is perfect;
\item if $\wh{Q}=1$ and $H_2(Q,\Z)=0$, then $N$ and $G$ are both perfect.
\end{enumerate}
\end{prop}
 
\begin{proof} With the exception of item (5), the proof is covered by 
Proposition 2.10 of  \cite{mb:gilb}. (The
crucial property of residual finiteness is due to Wise \cite{wise1, wise2}.)

For item (5), we consider the 5-term exact sequence extracted from the corner of the LHS spectral sequence for
$1\to N\to G\to Q\to 1$:
$$
H_2(Q,\Z) \to H_0(Q,\, H_1(N,\Z)) \to H_1(G,\Z) \to H_1(Q,\Z)\to 0.
$$
As all the other terms are zero, $H_0(Q,\, H_1(N,\Z))=0$.
By definition, $H_0(Q,\, H_1(N,\Z))$ is the group of coinvariants for the action of $Q$ on $ H_1(N,\Z)$ that is
induced by conjugation in $G$. As the abelian group $ H_1(N,\Z)$ is finitely generated, 
its automorphism group is residually finite.
As $Q$ has no non-trivial finite quotients, its action on $ H_1(N,\Z)$ must be trivial. Therefore
$H_1(N,\Z)=H_0(Q,\, H_1(N,\Z))=0$. 
\end{proof}

\section{Fibre products}

Our proof of Theorem \ref{t:main} relies on the various properties of fibre products that we establish in this section.
These properties cover three topics: the finiteness properties of fibre products, their behaviour with respect to
profinite completions, and their interaction with wreath products.
 
\subsection{Fibre products and finiteness properties}

For $i=1,\dots,d$, let  $p_i:G_i\to Q$ be a homomorphism of groups. The {\em fibre product} of this family
of maps is
$$
P_d = \{ (g_1,\dots,g_d) \mid p_i(g_i)=p_j(g_j),\ i,j=1,\dots,d\} < G_1\times\dots\times G_d.
$$
The case  $p_1=\dots=p_d$ will be of particular interest in this article.

$P_d$ is the preimage of the diagonal subgroup 
$$Q_d^\Delta := \{(q,\dots,q) \mid q\in Q\}< Q\times\dots\times Q$$
and there is a short exact sequence 
\begin{equation}\label{e1}
1\to N^{(d)} \to P_d \to Q_d^\Delta\to 1
\end{equation}
where $N_i= \ker p_i$ and $N^{(d)}= N_1\times\dots\times N_d$.


We need a criterion to ensure that $P_d$ is finitely presented; we will deduce this from the following 
{\em Asymmetric 1-2-3 Theorem}
\cite{bhms}. 

\begin{theorem}[\cite{bhms}] \label{t:123}
For $i=1,2$, let $1\to N_i\to G_i\overset{p_i}\to Q\to 1$ be a short exact sequence of groups. If $G_1$ 
and $G_2$ are finitely presented, $Q$ is of type $F_3$, and 
at least one of the groups $N_1, N_2$  is finitely
generated, then the fibre product $P<G_1\times G_2$ is finitely presented.
\end{theorem} 

\begin{corollary}\label{Pfp} Suppose $d\ge 2$ and
  let $1\to N_i\to G_i\overset{p_i}\to Q\to 1$ be a short exact sequence of groups, for $i=1,\dots,d$. If 
the groups $G_i$ are all finitely presented, the groups $N_i$ are finitely generated, and $Q$ is of type $F_3$,
 then the associated fibre product $P_d<G_1\times\dots\times G_d$ is finitely presented.
\end{corollary}

\begin{proof}
We proceed by induction on $d$; the case $d=2$ is covered by the theorem. 
Let $P_d < G_1\times\dots\times G_d$ be the fibre product of $p_1,\dots,p_d$. For the
inductive step, first note that  
$$P_d \ < \ P_{d-1}\times G_d \ < \ G_1\times\dots\times G_d$$ 
is the fibre product of the map $p_d:G_d\to Q$ and the composition $P_{d-1}\to Q_{d-1}^\Delta\to Q$,
where $P_{d-1}\to Q_{d-1}^\Delta$ is the map from (\ref{e1}) and $Q_{d-1}^\Delta\to Q$ is
the isomorphism $(q,\dots,q)\mapsto q$. To complete the proof, we apply the theorem,
noting that $P_{d-1}$ is finitely presented, by induction.
\end{proof}

We shall also need the following more elementary result.

\begin{lemma}\label{l:fg} For $i=1,\dots,d$, let $G_i\twoheadrightarrow Q$ be an epimorphism of groups.
If the groups $G_i$ are finitely
generated and $Q$ is finitely presented, then  the fibre product $P_d<G_1\times\dots\times G_d$ is finitely generated.
\end{lemma}

\begin{proof}
As in the preceding proof, induction reduces us to the case $d=2$. We fix a finite presentation $Q=\<a_1,\dots,a_n
\mid r_1,\dots, r_m\>$ and for $i=1,2$ choose  $a_{ij}\in G_i$ such that $p_i(a_{ij})=a_j$. We then add a finite set of
elements $B_i\subset \ker p_i$ to obtain a finite generating set for $G_i$, 
and denote by $\rho_{ik}$ the word obtained from $r_k$ by replacing each $a_j$ with  $a_{ij}$.
It is easy to check that the fibre product $P<G_1\times G_2$ is 
generated by 
$$
\{ (b_{1s},1), \ (1,b_{2s}),\ (a_{1j}, a_{2j}),\ (\rho_{1k}, 1) \mid j=1,\dots, n; \ k=1,\dots, m;\  b_{is}\in B_i\}.
$$
\end{proof}

\subsection{Fibre products and Grothendieck pairs} 

The idea of constructing Grothendieck pairs using fibre products originates in the work of Platonov and Tavgen \cite{PT}
and was extended in  \cite{BL}, \cite{BG} and \cite{mb:jems}.  

The following result is Lemma 2.2 in \cite{BG}.

\begin{lemma}\label{l:PT} Let $1\to N\to G\to Q\to 1$ be an exact sequence of finitely generated groups.
If $\wh{Q}=1$ and $H_2(Q,\Z)=0$, then  $N\hookrightarrow G$ 
induces an isomorphism of profinite completions.
\end{lemma}

The following variant of the Platonov-Tavgen argument will be useful; this is Theorem 2.2 of \cite{mb:jems}.

\begin{proposition} \label{p:PT} Let $p_1: G_1\to Q$ and $p_1: G_2\to Q$ be epimorphisms with $G_1$ and $G_2$
finitely generated and $Q$  finitely presented. Let $P<G_1\times G_2$
be the associated fibre product. If $\wh{Q}=1$  
and $H_2(Q,\Z)=0$,  then 
$P\hookrightarrow G_1\times G_2$ induces an isomorphism of profinite completions.
\end{proposition}  

We need an extension to the case of $d\ge 2$ factors.

\begin{theorem}\label{t:fib}
For $i=1,\dots,d$, let  $p_i:G_i\to Q$ be an epimorphism of finitely generated groups, and  let $P_d<
\G:=G_1\times\dots\times G_d$
be the associated fibre product.
If $Q$  is finitely presented, $\wh{Q}=1$  
and $H_2(Q,\Z)=0$,  then 
$P_d\hookrightarrow \G$ induces an isomorphism of profinite completions.
\end{theorem}

\begin{proof}
We argue by induction on $d$, as in the proof of Theorem \ref{t:123}. In the inductive step, we appeal 
to Lemma \ref{l:fg} to ensure that $P_{d-1}$ is finitely generated. We can then apply Proposition \ref{p:PT}
to $p_d:G_d\to Q$ and $P_{d-1}\to Q_{d-1}^\Delta\cong Q$, noting that $P_d$ is the fibre product of these maps.
\end{proof}

\subsection{Fibre products and wreath products}\label{s:fibW}

Given groups $A$ and $B$, with $B$ finite, the  {\em wreath product} $A\wr B$ is the semidirect product $A^B\rtimes B$, or
more precisely
$(\oplus_{b\in B} A_b)\rtimes B$, with fixed isomorphisms $\mu_b:A\to A_b$ so that $b\in B$ acts on $A_{b'}$
as $\mu_{bb'}\circ\mu_{b'}^{-1}$.   
We identify $A^\Delta < A\times\dots\times A$ with its image  under $(\mu_b)_{b\in B}$. The following
trivial observation will play an important role in what follows.

\begin{lemma}\label{l:diag}
$\<A^\Delta, B\> < A\wr B$ is the direct product $A^\Delta \times B\cong A\times B$.
\end{lemma} 

Given $B$ and a short exact sequence of groups $1\to N\to G\to Q\to 1$, we take the
direct product of $|B|$ copies of the sequence,
 indexed by the elements of $B$, and let $B$ permute these copies by its left action on the indices.
The resulting semidirect products give us a (non-exact) sequence of groups
$$
N\wr B \hookrightarrow G\wr B \twoheadrightarrow Q\wr B.
$$  
The action of $B$ preserves the fibre product $P_B< G^B=\oplus_{b\in B} G_b$ of the maps $G_b\to Q_b$,
giving a semidirect product 
$$P_B\rtimes B = \< P_B, B\> < G\wr B$$
and a (non-exact) sequence of groups
$$
N\wr B \hookrightarrow P_B\rtimes B  \twoheadrightarrow \<Q^\Delta, B\> < Q\wr B.
$$  
From Lemma \ref{l:diag} we deduce:
\begin{lemma}\label{l:ontoQ} With the notation established above, there is 
surjection $$P_B\rtimes B\twoheadrightarrow Q^\Delta \cong Q.$$
\end{lemma} 
 
\section{Fixed point criteria}

In this section we present criteria that guarantee fixed points for group actions on complete \cat spaces of finite
dimension. These criteria are extracted from the more general criteria explained in \cite{mb:helly}
and \cite{mb:mcg}. 

The following result is a special case of \cite[Corollary 3.6]{mb:mcg}.

\begin{proposition}
Let $d$ be a positive integer and let $X$  be a complete $\cat$ space of dimension less than $d$. 
Let $S_1,\dots, S_d\subset{\rm{Isom}}(X)$  be conjugates of a subset $S\subset{\rm{Isom}}(X)$ with $[s_i,s_j]=1$ for all $s_i\in S_i$ and $s_j\in S_j$.
If every element of $S$ (hence $S_i$) has a fixed point in $X$, then so does every finite subset of $S$ (hence $S_i$).
\end{proposition}

\begin{corollary}\label{c:ell}
Let $d$ be a positive integer, let $X$  be a complete $\cat$ space of dimension less than $d$,
and let  $H_1,\dots, H_d <{\rm{Isom}}(X)$  be subgroups that are pairwise conjugate.  
If each $H_i$ is generated by a finite set of elliptic elements, then $D=\<H_1,\dots,H_d\>$ has a fixed point in $X$.
\end{corollary}

\begin{proof}
Let $S=S_1$ be a finite set of elliptics generating $H_1$. We conjugate $S$ to obtain a generating set $S_i$ for 
each $H_i$. The proposition says that $\Fix(S_i)=\Fix(H_i)$
is non-empty, whence  $\Fix(D)$ is non-empty, by Lemma \ref{c2.5}.
\end{proof}

For $n\in\N$, an {\em $n$-flat} in a metric space
 $X$ is an isometrically embedded copy of  Euclidean space $\mathbb{E}^n\hookrightarrow X$.

\begin{lemma}
 If $K_0,\dots,K_d$ are groups  with ${\rm{Hom}}(K_i,\R)=0$ and $X$ is a 
complete \cat space that does not contain any $(d+1)$-flats, then there does not exist an action
$\rho: K_0\times\dots\times K_d \to {\rm{Isom}}(X)$ such that each $\rho(K_i)$ contains a hyperbolic isometry.
\end{lemma}

\begin{proof}
We shall prove the lemma by induction, the case $d=0$ being trivial. Assume that the lemma is true for $d'\le d-1$.
The induction will be complete if we can derive a contradiction from the assumption 
that there are hyperbolic isometries $\gamma_i\in\rho(K_i)$ for $i=0,\dots,d$. If this were the case,
then, according to Proposition \ref{p:min}, the subspace $\Min (\gamma_0)$ would split isometrically  as 
$Y\times \R$ and  the centraliser $C(\gamma_0)$ of $\gamma_0$ in $\isom(X)$ would
 preserve $\Min (\gamma_0)$ and its splitting, acting by translations on the second
factor of $Y\times\R$. The group of translations is $\mathbb{R}$ and ${\rm{Hom}}(K_i,\R)=0$,
so  $K_1\times\dots\times K_d$ must act trivially on the second factor. Thus we obtain an action of
$K_1\times\dots\times K_d$ on $Y_0=Y\times\{0\}$. Part (1) of Proposition \ref{p:min} assures us that
$Y_0\subset X$ is closed and convex, hence a \cat space, and part (4) tells us that   $\gamma_i\in K_i$
acts as a hyperbolic isometry of $Y_0$,  for $i=1,\dots,d$. 
But $Y\times \R = \Min (\gamma_0)$ embeds isometrically in $X$, so $Y_0$ does not contain a $d$-flat.
This contradicts our inductive hypothesis.
\end{proof}

\begin{theorem} \label{t:fix}
Let $G$ be a group and suppose that there is a 
subgroup $D=H_0\times\dots\times H_d < G$ with $H_i$  conjugate to $H_0$ in $G$ for $i=1,\dots,d$. If $H_0$
is finitely generated and has finite abelianisation, then $D$ has a fixed point whenever $G$ acts by 
semisimple isometries on a complete \cat space of dimension at most $d$.
\end{theorem}

\begin{proof} The hypothesis $\dim (X)\le d$ is stronger than requiring that
$X$ contains no $(d+1)$-flat, so the preceding lemma tells us that there are no hyperbolic elements in
any of the subgroups of $H_i$. Because $H_0$ is finitely generated, Corollary \ref{c:ell}  completes the proof.  
\end{proof}

The following result was stated as Theorem B in the introduction.

\begin{corollary}\label{c:fixd}
If $A$ is a finitely generated group with finite abelianisation and $B$ is a finite group, then
$A\wr B$ has $\Fix_d$ where $d= |B|-1$.
\end{corollary}

\begin{proof} Let $G=A\wr B= (\oplus_{b\in B}A_b)\rtimes B $ and $D=\oplus_{b\in B}A_b$. 
Theorem \ref{t:fix}   tells us that $D$ has a fixed point whenever $A\wr B$ acts by semisimple isometries on a 
complete \cat space $X$ with $\dim(X)\le |B|-1$. Because $B<A\wr B $ normalises $D$, it leaves
its set of fixed points $\Fix(D)\subset X$ invariant.  $\Fix(D)$ is closed and convex, hence a complete
\cat space. Proposition \ref{l:finite} provides a point in $\Fix(D)$ that is fixed by $B$ and hence by $A\wr B=\<D,B\>$.
\end{proof}

\section{Proof of Theorem \ref{t:main}}
Let $Q$ be a group satisfying the conditions listed in Section \ref{s:groups}.  
By Proposition \ref{p:rips}, there is
a short exact sequence
$$
1\to N \to G\to Q\to 1
$$
with $N$ finitely generated and perfect, and $G$ finitely presented, residually finite and perfect. Given $d\ge 2$, we fix
a finite group $B$ with $|B|=d+1$. Proceeding as in Section \ref{s:fibW}, we take the direct product of $|B|$ copies
of this sequence, indexed by the elements of $B$, and take the fibre product of the maps $G_b\to Q$ to obtain
$$N^B\hookrightarrow P_B \hookrightarrow G^B.$$
The action of $B$ permuting the direct factors of $G^B$ leaves $N^B$ and $P_B$ invariant, so the
above inclusions extend to
$$
N\wr B \overset{i}\hookrightarrow P_B\rtimes B  \overset{j}\hookrightarrow G\wr B.
$$
We claim that this  triple of groups has the properties required in Theorem~\ref{t:main}. 

Towards showing that $i$  induces an
isomorphism of profinite completions, note first that Lemma \ref{l:PT} applies to $N^B\hookrightarrow P_B$,
because $N^B$ is normal in $P_B$ with quotient $Q$. Likewise, Theorem \ref{t:fib} assures us that $P_B\hookrightarrow G^B$, the
restriction of $j$, induces an isomorphism of profinite completions. The action of $B$ permuting the factors of $G^B$ 
extends to $\wh{G^B}$, where it preserves the dense subgroups $P_B$ and $N^B$. Noting that
$\wh{N}\wr B = \wh{N\wr B}$ and $\wh{G}\wr B = \wh{G\wr B}$, we conclude that 
$\wh{i}$ and $\wh{j}$ are isomorphisms extending
$\wh{N^B}\to\wh{P_B}$ and $\wh{P_B}\to \wh{G^B}$. This establishes Theorem \ref{t:main}(1).

$N\wr B$ is finitely generated, since $N$ is. Corollary \ref{Pfp} assures us that $P_B$ is finitely presented,
whence the finite extension $P_B\rtimes B$ is too. By construction, $Q$ is of typre $F_\infty$, and therefore so 
is $Q\wr B$.
(And we could take it to be virtually of type $\F$ if desired.) This establishes Theorem \ref{t:main}(2).

$N$ and $G$ are finitely generated and perfect, so Corollary \ref{c:fixd} tells us that $N\wr B$ and $G\wr B$ have $\Fix_d$,
since $|B|=d+1$. In contrast, $P_B\rtimes B$ maps onto $Q$, as in Lemma \ref{l:ontoQ}, and therefore it is a non-trivial
amalgamated free product -- in particular it does not have property FA or $\Fix_d$.
\qed

\section{Flexibility and a decision problem}\label{s:last}

It is clear from the discussion in Section \ref{s:groups} that there is a great deal of flexibility in how one chooses the
input groups $Q$. Consequently, one is free to impose various extra conditions on the Grothendieck pairs 
$P\rtimes B\hookrightarrow G\wr B$ that we have constructed. In particular,
 the range of pairs that one obtains is sufficient to accommodate many of the undecidability phenomena described in
\cite{mb:karl} and elsewhere. For example, by following the proof of \cite[Theorem B]{mb:karl} we obtain the
following theorem. Similar results hold with condition $\Fix_d$ in place of FA.

\begin{theorem}\label{t:iso}
There does not exist an algorithm that, given a 
finitely presented, residually finite group $\G$ that has property FA and a finitely
presentable subgroup $u:P\hookrightarrow\G$ with
$\hat u:\hat P\to \hat \G$ an isomorphism, can determine whether or
not $P$ has property FA.
\end{theorem}

\begin{proof}  As in \cite{mb:karl}, one can enhance the groups constructed in \cite{CM} to obtain
a recursive sequence of finite presentations $\mathcal{Q}^{(m)} \equiv \< S \mid R^{(m)}\>$ for groups $Q^{(m)}$,
with $S$ and $|R^{(m)}|$ fixed, so that (i)
there is no algorithm to determine which of the groups are trivial, but (ii) if $Q^{(m)}\neq 1$ then it 
satisfies the properties listed 
in Section \ref{s:groups}. We apply the algorithm of Proposition \ref{p:rips} to the presentations $\mathcal{Q}^{(m)}$
to obtain $G^{(m)}\twoheadrightarrow Q^{(m)}$,
with an explicit presentation for $G^{(m)}$ and hence $G^{(m)}\times G^{(m)}$. The fibre product
$P^{(m)}\hookrightarrow G^{(m)}\times G^{(m)}$
 is given by the finite generating described in Lemma \ref{l:fg}, with $B_i$ the
given relators of $G^{(m)}$. Theorem \ref{t:123} assures us that $P^{(m)}$ is finitely presentable.
We pass from $P^{(m)}\hookrightarrow G^{(m)}\times G^{(m)}$
 to $u_m: P^{(m)}\rtimes (\Z/2)\hookrightarrow G^{(m)}\wr (\Z/2)$ and
then argue as in the proof of Theorem \ref{t:main} to see that $u_m$ induces an isomorphism of profinite completions,
that $G^{(m)}\wr (\Z/2)$ has property FA (in fact $\Fix_1$),
and that $ P^{(m)}\rtimes (\Z/2)$ maps onto $Q^{(m)}$. 

Note that the groups $G^{(m)}\wr (\Z/2)$ are given by a recursive sequence of presentations and the maps $u_m$ are given by a recursive sequence of generating sets for the subgroups $P^{(m)}\rtimes (\Z/2)$.

If $Q^{(m)}\neq 1$ then 
$ P^{(m)}\rtimes (\Z/2)$ does not have FA, since it maps onto $Q^{(m)}$. 
But if $Q^{(m)}= 1$ then $u_m$ is an isomorphism, so $ P^{(m)}$ does  have FA. And by
construction, there is no algorithm to decide which of these alternatives holds. 
\end{proof}

\end{document}